\documentclass{amsart}
\usepackage{amsmath,amsthm}
\usepackage{amsfonts,amssymb}
\usepackage{accents}
\usepackage{enumerate}
\usepackage{accents,color}
\usepackage{graphicx}

\newcommand{\lvt}{\left|\kern-1.35pt\left|\kern-1.3pt\left|}
\newcommand{\rvt}{\right|\kern-1.3pt\right|\kern-1.35pt\right|}

\newtheorem{thm}{Theorem}[section]
\newtheorem{cor}[thm]{Corollary}
\newtheorem{lem}[thm]{Lemma}
\newtheorem{prop}[thm]{Proposition}
\newtheorem{exam}[thm]{Example}

\theoremstyle{remark}
\newtheorem{rem}{Remark}[section]

 \def\la{{\langle}}
 \def\ra{{\rangle}}

 \def\d{\mathrm{d}}

 \def\sph{{\mathbb{S}^{d-1}}}

 \def\sh{{\mathsf h}}
 \def\sm{{\mathsf m}}
 \def\sw{{\mathsf w}}

 \def\sQ{{\mathsf Q}}
 \def\sS{{\mathsf S}}

 \def\fD{{\mathfrak D}}

 \def\a{{\alpha}}
 \def\b{{\beta}}
 \def\g{{\gamma}}

 \def\l{{\lambda}}
 \def\o{{\omega}}
 \def\s{\sigma}
 \def\la{{\langle}}
 \def\ra{{\rangle}}

 \def\CD{{\mathcal D}}
 
 \def\CH{{\mathcal H}}

 \def\CU{{\mathcal U}}
 \def\CV{{\mathcal V}}

 \def\NN{{\mathbb N}}

 \def\RR{{\mathbb R}}
 \def\SS{{\mathbb S}}
 \def\VV{{\mathbb V}}

      \def\proj{\operatorname{proj}}

\def\lla{\langle{\kern-2.5pt}\langle}
\def\rra{\rangle{\kern-2.5pt}\rangle}

\newcommand{\wt}{\widetilde}
\newcommand{\wh}{\widehat}

\def\f{\frac}

\graphicspath{{./}}
\begin{document}

\title{Sobolev orthogonal polynomials on the conic surface}
\author[L. Fern\'andez]{Lidia Fern\'andez}
\address{Instituto de Matem\'aticas IMAG \&
Departamento de Matem\'{a}tica Aplicada. Universidad de Granada (Spain)}
\email{lidiafr@ugr.es}

\author[T. E. P\'erez]{Teresa E. P\'erez}
\address{Instituto
de Matem\'aticas IMAG \&
Departamento de Matem\'{a}tica Aplicada. Universidad de Granada (Spain)}
\email{tperez@ugr.es}

\author[M. A. Pi\~nar]{Miguel Pi\~nar}
\address{Instituto
de Matem\'aticas IMAG \&
Departamento de Matem\'{a}tica Aplicada. Universidad de Granada (Spain)}
\email{mpinar@ugr.es}
\author[Y. Xu]{Yuan~Xu}
\address{Department of Mathematics, University of Oregon, Eugene,
OR 97403--1222, USA}
\email{yuan@uoregon.edu}
\thanks{
First (LF), second (TEP) and third (MAP) authors thank grant PID2023.149117NB.I00; LF and TEP also thank CEX
2020-001105-M, and LF thanks PID2024-155133NB-I00; all funded by MICIU/AEI/10.13039/501100011033 and ERDF. The fourth author (YX) was partially supported by 
 Simons Foundation Grant \#849676}
\date{\today}

\begin{abstract}
Orthogonal polynomials with respect to the weight function $\sw_{\b,\g}(t) = t^\beta (1-t)^\gamma$,
$\gamma > -1$, on the conic surface $\{(x,t): \|x\| = t, \, x \in \mathbb{R}^d, \, t \le 1\}$ are studied recently, and
are shown to be eigenfunctions of a second order differential operator $\CD_\gamma$ when $\beta =-1$.
We extend the setting to the Sobolev inner product, defined as the integration of the $s$-th normal derivative
$\fD = \frac{\d}{\d t} - t^{-1} \la x, \nabla_x\ra$ of the cone with respect to $w_{\beta+s,0}$ over the conic surface, 
plus a sum of integrals over the rim of the cone. Our main results provide an explicit construction of an 
orthogonal basis and a formula for the orthogonal
projection operators; the latter is used to exploit the interaction of differential operators and the projection operator,
which allows us to study the convergence of the Fourier orthogonal series. The study can be regarded as an
extension of the orthogonal structure to the weight function $w_{\beta, -s}$ for a positive integer $s$. It shows,
in particular, that the Sobolev orthogonal polynomials are eigenfunctions of $\CD_{\gamma}$ when $\gamma = -1$.
\end{abstract}

\maketitle

\section{Introduction}
\setcounter{equation}{0}

Orthogonal polynomials on the conic surface of revolution were studied recently, which are shown
to possess properties parallel to those of spherical harmonics on the unit sphere. 
Let $\VV_0^{d+1}$ be the conic surface
$$
  \VV_0^{d+1} = \{(x,t) \in \RR^{d+1}: \, \|x\| = t, \, x \in \RR^d, \, 0 \le t \le 1\}
$$
in $\RR^{d+1}$. For the weight function $\sw_{\b,\g}(t) = t^\b(1-t)^\g$, $\b > -d$ and $\g > -1$, the orthogonal
polynomials with respect to the inner product
$$
  \la f,g\ra_{\b,\g} = b_{\b,\g} \int_{\VV_0^{d+1}} f(x,t) g(x,t) \sw_{\b,\g}(t) \d \sm (x,t)
$$
where $\d \sm$ is the Lebesgue measure on the conic surface, are called the Jacobi polynomials on the cone. These polynomials are studied in \cite{OX20, X20, X21a, X21, X23}.
It was shown in \cite{X20} that these polynomials share many
properties of spherical harmonics, including an explicit orthogonal basis and an addition formula, which provides
essential tools for an extensive study in approximation theory and computational analysis over the cone in
\cite{X21}. In particular, they are used to study the best polynomial approximation in \cite{GX}. Another remarkable 
property of the Jacobi polynomials on the cone is that they are eigenfunctions of a second-order linear differential 
operator $\CD_\g$ when $\b = -1$, which is an analog of the Laplace-Beltrami operator on the unit sphere. 

The purpose of the present paper is to study Sobolev orthogonal polynomials on the conic surface, which
are orthogonal with respect to an inner product that contains derivatives. The first case is
$$
  \la f,g\ra_{\b, -1} =  \frac{1}{\o_d}
       \int_{\VV_0^{d+1}} \fD f (x,t)\fD g(x,t) t^{\b+1} \d \sm(x,t)
          +  \frac{\l}{\o_d} \int_{\SS^{d-1}}f(\xi,1) g(\xi,1) \d \s (\xi),
$$
where $\fD = \f{\d}{\d t} - t^{-1} \la x, \nabla_x\ra$ is the derivative in the normal direction of the cone, and we also 
consider $\la f,g\ra_{\b, -s}$ that involves derivatives up to $s$ order for a positive integer $s$. 
 Like in the case of $\g > -1$, our main result provides an explicit construction of orthogonal bases and a closed-form 
 formula for the orthogonal projection operator. The study requires an
extension of the Jacobi polynomials with a parameter being a negative integer, which needs to satisfy the
Sobolev orthogonality of one variable that is inherited from $\la \cdot,\cdot\ra_{\b,-s}$ when we restrict the
inner product to polynomials depending only on the $t$ variable. Several authors have studied such 
Sobolev orthogonal polynomials of one variable; see, for example, \cite{AAR02, MPP, GW, Sh1, Sh2, X19}
and \cite{MX}. We shall follow the approach in \cite{X19} since it is more convenient for studying orthogonal
projection operators and provides a link, in particular, between the Sobolev orthogonal structure and
the ordinary orthogonal structure, which is useful for studying the convergence of the Fourier orthogonal
series in the Sobolev orthogonal polynomials. In the framework of polynomial approximation theory on the ball
and standard or Sobolev orthogonal polynomials, we can refer to \cite{BF21, DX11, DaiX, F17a, F17b, LX14, PX18, W}, 
among others.  

In more than one way, our study extends the Jacobi polynomials for $\sw_{\b,\g}$ on the cone from
$\g > -1$ to $\g = -s$ with $s \in \NN$. We will show, in particular, that the spectral operator $\CD_{-s}$
has the Sobolev orthogonal polynomials as eigenfunctions if $s = 1$. While the latter fails for $s > 1$,
we do have a clear understanding of what the eigenspaces of $\CD_{-s}$ are. For orthogonal polynomials
in several variables, our study is also closely related to the Sobolev orthogonal polynomials on the unit ball,
which have been extensively studied (see \cite{DX} and its references therein). In particular, the description
of the eigenspaces of $\CD_\g$ is similar to the study on the unit ball in \cite{PX09}.

The paper is organized as follows. The next section is preliminary, in which we recall two essential ingredients
needed for our study, the Jacobi polynomials with negative parameters and spherical harmonics. In Section 3
we review results on ordinary orthogonal polynomials on the conic surface and discuss further properties of
the orthogonal projection operators. The Sobolev orthogonal polynomials are defined and studied in Section 4.
Finally, the eigenspaces of the operator $\CD_\g$ are discussed in Section 5.

\section{Preliminary}
\setcounter{equation}{0}

The study for orthogonal polynomials on the conic surface follows that of spherical harmonics on the unit
sphere. The latter will also be essential for constructing an orthogonal basis on the cone.
Another ingredient is the Jacobi polynomial, which we often need an extension to negative parameters
in the study of the Sobolev orthogonal polynomials.

\subsection{Jacobi polynomials with negative parameters}
The Jacobi polynomials $P_n^{(\a,\b)}$ are given explicitly by the hypergeometric function
\begin{equation}\label{eq:2F1Jacobi}
P_n^{(\a,\b)}(t) = \frac{(\a+1)_n}{n!} {}_2F_1 \left( \begin{matrix}-n, n+\a+\b+1 \\ \a+1\end{matrix}; \frac{1-x}{2}\right)
\end{equation}
for $\a,\b > -1$ and $n = 0, 1, 2, \ldots$. They are orthogonal with respect to the weight function
$w_{\a,\b}(t) = (1-t)^\a (1+t)^\b$ on $[-1,1]$ with $\a, \b > -1$, and they satisfy
$$
   \frac{c_{\a,\b}}{2^{\a+\b+1}}\int_{-1}^1 P_n^{(\a,\b)}(t) P_n^{(\a,\b)}(t) w_{\a,\b}(t) \d t = h_n^{(\a,\b)} \delta_{n,m},
$$
where $c_{\a,\b}$ is the constant so that $h_0^{(\a,\b)} =1$,
\begin{equation}\label{JacobiNorm}
  c_{\a,\b} = \frac{\Gamma(\a+\b+2)}{\Gamma(\a+1)\Gamma(\b+1)} \quad\hbox{and}\quad  h_n^{(\a,\b)}
    = \frac{(\a+1)_n(\b+1)_n (\a+\b+n+1)}{n!(\a+\b+2)_n (\a+\b+2n+1)}.
\end{equation}

For studying the Sobolev orthogonal polynomials, we often need the parameters $\a$ or $\b$ to be
negative integers. Such polynomials are discussed already in \cite{Sz}, but they are no longer orthogonal
with respect to $w_{\a,\b}$, since the weight function $(1-t)^a(1+t)^{b}$ is no longer integrable if $a$ or
$b \le -1$. Moreover, $P_n^{(\a,\b)}$ has a degree reduction if $n+\a+\b$ is a negative integer between
$1$ to $n$, which causes problems for studying the Sobolev orthogonal polynomials, especially in several
variables, since such polynomials are needed for all $n \in \NN_0$.

What we need in this paper are the polynomials $P_n^{(\a,-s)}$ with $\a > -1$ and $s \in \NN$. These
polynomials are well defined if $n \ge s$ and satisfy \cite[Section 4.22]{Sz}
\begin{equation}\label{eq:Jacobi_n<0}
    P_n^{(\a,-s)} (t) = \frac{(-\a-n)_s}{2^s (-n)_s} (1+t)^s P_{n-s}^{(\a,s)}(t), \quad n \ge s,
\end{equation}
which follows from \cite[(4.22.2)]{Sz} by using $P_n^{(\a,\b)}(-t) = (-1)^n P_n^{(\b,\a)}(t)$. The definition of
$P_n^{(\a,-s)}$ for $n < s$ could be problematic because of the degree reduction. Such polynomials have
been studied in the setting of the Sobolev orthogonal polynomials; for example, the Sobolev orthogonality
defined via the inner product
\begin{equation} \label{eq:SOPd=1}
   [f,g]_{\a,\b}^{-s} : = \int_{-1}^1 f^{(s)}(t) g^{(s)}(t) w_{\a,\b}(t) \d t + \sum_{k=0}^{s-1} \mu_k f^{(k)}(1)g^{(k)}(1),
\end{equation}
where $\mu_k$ are fixed positive constants and $\a ,\b > -1$. The study of such polynomials and their
orthogonality has appeared in several papers; see, for example, \cite{AAR02, GW, MPP, Sh1, Sh2, X17, X19}
and \cite{MX}. There are several ways to define a complete set of orthogonal polynomials for the inner
product \eqref{eq:SOPd=1}. We shall follow the approach given in \cite{X19}, see also \cite{Sh1, Sh2}, which
is more suitable for studying the Fourier orthogonal series. We now recall the necessary result from \cite{X19}.

For convenience, we first define a renormalization of the Jacobi polynomials,
\begin{equation} \label{eq:JacobiJ}
  \wh P^{(\alpha,\beta)}_{n}(t) = A_n^{(\a,\b)} {P}^{(\alpha,\beta)}_{n}(t) \quad\hbox{with}\quad
   A_n^{(\a,\b)} =  \frac{2^n}{(n+\a+\beta+1)_{n}}
\end{equation}
for $\a,\b> -1$. This normalization has the advantage that it satisfies, by \cite[(4.5.5)]{Sz},
\begin{align} \label{eq:diffJ}
  \frac{d}{dt} \wh P_n^{(\a,\b)}(t) =  \wh P_{n-1}^{(\a+1,\b+1)}(t).
\end{align}
Now, for $\a,\b > -1$, $s \in \NN$ and $n = 0,1,2,\ldots,$ we define a new sequence of polynomials
\begin{align} \label{eq:CJ}
\begin{split}
   J_n^{(\a-s,\b-s)}(t) := \begin{cases}  \dfrac{(t+1)^n}{n!}, & 0 \le n \le s-1, \vspace{.05in} \\
        \displaystyle{\int_{-1}^t \frac{(t-u)^{s-1}}{(s-1)!} \wh P_{n-s}^{(\a,\b)}(u) \d u}, & n \ge s.
   \end{cases}
\end{split}
\end{align}
It is easy to see that $J_n^{(\a-s,\b-s)}$ is a polynomial of degree $n$ and it satisfies
\begin{align}
 \partial^s J_n^{(\a-s,\b-s)}(t) &\ = \wh P_{n-s}^{(\a,\b)}(t), \qquad  n \ge s; \label{eq:CJ1a}\\
 \partial^k J_n^{(\a-s,\b-s)}(-1) &\ = \begin{cases} \delta_{k,n}, & n \le s-1, \\
         0, &  n \ge s,
         \end{cases} \qquad 0\le k \le s-1, \label{eq:CJ1b}
\end{align}
where $\partial^k$ denotes the $k$-th derivative. These are our polynomials that extend the definition of
the Jacobi polynomials to allow negative parameters, which are also orthogonal polynomials with respect to 
the Sobolev inner product \eqref{eq:SOPd=1}. More precisely, we have the following \cite{X19}:

\begin{thm}\label{thm:SOP-1d}
For $\a,\b > -1$ and $s \in \NN$. The polynomial $J_n^{(\a-s,\b-s)}$ is orthogonal with respect to the
inner product $\left[ \cdot, \cdot \right]_{\a,\b}^{-s}$ and its norm square is given by
$$
 \left[J_n^{(\a-s,\b-s)},J_n^{(\a-s,\b-s)}\right]_{\a,\b}^{-s}
   = \begin{cases}\mu_n &  0 \le n \le s-1 \\   \wh h_{n-s}^{(\a,\b)} & n \ge s \end{cases},
$$
where $\mu_n$ comes from \eqref{eq:SOPd=1}, and $\wh h_n^{(\a,\b)}$ is the norm square of $\wh P_n^{(\a,\b)}$, which is given in terms of $h_n^{(\a,\b)}$ by
$$
\wh h_n^{(\a,\b)} = \frac{2^{\a+\b+1}}{c_{\a,\b}} \left(A_n^{(\a,\b)}\right)^2 h_n^{(\a,\b)}.
$$
\end{thm}

Our next proposition shows that, if $\a > -1$ and $s\in \NN$, then the definition $J_n^{(\a,-s)}$ in
\eqref{eq:CJ} agrees with that of \eqref{eq:Jacobi_n<0} when $n\ge s$.

\begin{prop}
For $\a > -1$ and $s \in \NN$,
\begin{equation}\label{eq:Jn(a,-s)}
  J_n^{(\a,-s)} (t) = \frac{(n-s)!}{n!}  (1+t)^s \wh P_{n-s}^{(\a,s)}(t), \qquad n \ge s.
\end{equation}
In particular, for $n \ge s$,
\begin{equation}\label{eq:Jn(a,-s)=Pn}
  J_n^{(\a,-s)} (t) = \frac{(-1)^s 2^s}{(-\a-n)_s}A_{n-s}^{(\a,s)} P_n^{(\a,-s)}(t).
\end{equation}
\end{prop}

\begin{proof}
We use the hypergeometric expression of the Jacobi polynomials,
$$
P_n^{(\a,\b)}(t) = \frac{(\a+1)_n}{n!} {}_2F_1\left(\begin{matrix} -n \,\,\, n+\a+\b+1\\ \a+1 \end{matrix}; \frac{1-t}{2} \right)
$$
and the fact that $P_n^{(\a,\b)}(t) = (-1)^n P_n^{(\b,\a)}(-t)$. It follows that
\begin{align*}
 &  \int_{-1}^t \frac{(t-u)^{s-1}}{(s-1)!} P_n^{(\a+s,0)}(u) \d u
   = (-1)^n  \int_{-1}^t \frac{(t-u)^{s-1}}{(s-1)!} P_n^{(0, \a+s)}(-u) \d u \\
  & \qquad\quad = (-1)^n \sum_{k=0}^n \frac{(-n)_k (n+\a+s+1)_k}{k! k!}
   \int_{-1}^t \frac{(t-u)^{s-1}}{(s-1)!} \left(\frac{1+u}{2} \right)^k \d u \\
  & \qquad\quad = \frac{(-1)^n}{s!} (1+t)^s \sum_{k=0}^n \frac{(-n)_k (n+\a+s+1)_k}{k! (s+1)_k}  \left(\frac{1+t}{2} \right)^k \\
  & \qquad\quad = \frac{(-1)^n}{s!} (1+t)^s \frac{n!}{(s+1)_n} P_n^{(s,\a)}(-t)
     =   \frac{n!}{(n+s)!}  (1+t)^sP_n^{(\a,s)}(t).
 \end{align*}
Replacing $n$ by $n-s$ and using $A_n^{(\a,s)} = A_n^{(\a+s,0)}$, the identity \eqref{eq:Jn(a,-s)} then follows from
\eqref{eq:CJ}. The second identity follows from the identity \eqref{eq:Jacobi_n<0}.
\end{proof}

It should be pointed out that the integral expression of $J_n^{(\a,-s)}$ for $n \ge s$ in \eqref{eq:CJ} is more
convenient for studying the Fourier orthogonal series, as shown in \cite{X19} and as we shall see in
Section 4 below.

\subsection{Spherical harmonics}
A homogeneous polynomial $Y$ of $d$ variables is called a solid harmonic if $\Delta Y =0$, where $\Delta$
is the Laplace operator on $\RR^d$. We denote by $\CH_m^{d,0}$ the space of homogeneous solid
harmonics of degree $m$ in $d$ variables. Thus, if $Y \in \CH_m^{d,0}$, then $Y(r \xi) = r^m Y(\xi)$ for
$\xi \in \sph$. Spherical harmonics are restrictions of solid harmonics on the unit ball. We denote the
space of spherical harmonics of degree $m$ by $\CH_m^d$. Thus, $\CH_m^d = \CH_m^{d,0} \vert_\sph$.
It is a common practice to identify $\CH_m^{d,0}$ and $\CH_m^d$, we distinguish them to emphasize the
dependence on variables for the orthogonal polynomials on the conic surface. It is well known that
$$
   \dim \CH_m^d =  \binom{m+d-2}{n}+\binom{m+d-3}{n-1},\quad m=1,2,3,\ldots,
$$
and spherical harmonics of different degrees are orthogonal with respect to the surface measure on the
unit sphere. Throughout the paper, we denote by $\{Y_\ell^m: 1 \le \ell \le \dim \CH_m^d\}$ an orthonormal
basis of $\CH_m^d$, so that
$$
  \frac{1}{\o_d} \int_\sph Y_\ell^n(\xi) Y_{\ell'}^m(\xi) \d\s (\xi) = \delta_{\ell,\ell'} \delta_{n,m},
$$
where $\d \s$ is the surface measure of $\sph$ and $\o_d$ is the surface area of $\sph$.

Let $\proj_n^\sph: L^2(\sph) \rightarrow \CH_n^d$ be the orthogonal projection operator from $L^2(\sph)$
onto $\CH_n^d$. If $\{Y_\ell^n: 1 \le \ell \le \dim \CH_n^{d-1} \}$ is an orthonormal basis of $\CH_n^d$, then
$$
  \proj_n^\sph f (x) = \sum_{\ell =1}^{\dim \CH_n^d} \hat f_{\ell}^n Y_\ell^n, \quad
   \hat f_\ell^n = \frac{1}{\o_d} \int_{\sph} f(\xi) Y_\ell^n(\xi) \d\s(\xi).
$$
For $f \in L^2(\sph)$, its Fourier expansion in spherical harmonics is defined by
$$
  f =  \sum_{n=0}^\infty \sum_{\ell =1}^{\dim \CH_n^d} \hat f_{\ell}^n Y_\ell^n =\sum_{n=0}^\infty  \proj_n^\sph f.
$$

Let $\Delta_0$ be the Laplace-Beltrami operator, which is the restriction of the Laplacian $\Delta$ on
the unit sphere. Under the spherical polar coordinates $x = r \xi$, $r > 0$, $\xi \in \sph$,
$$
  \Delta = \frac{\d^2}{\d r^2} + \frac{d-1}{r}  \frac{\d}{\d r} + \frac{1}{r^2} \Delta_0.
$$
The spherical harmonics are the eigenfunctions of $\Delta_0$. More precisely,
\begin{equation}\label{eq:LB-operator}
   \Delta_0 Y = - n (n+d-2)Y, \qquad Y \in \CH_n^d, \quad n=0,1,2,\ldots.
\end{equation}
The $\Delta_0$ is a second-order differential operator on the unit sphere. One can also consider
spherical gradient $\nabla_0$, the first-order differential operators, on the sphere, which is defined by
\begin{equation}\label{eq:nabla0}
     \nabla = \frac{1}{r} \nabla_0 + \xi \frac{\d}{\d r}, \quad x = r \xi, \quad \xi \in \sph.
\end{equation}
The integration by parts formula holds and gives \cite[(1.8.14)]{DaiX},
\begin{equation}\label{eq:int-sph}
   \int_{\sph} \nabla_0 f(\xi) \cdot  \nabla_0 g(\xi) \d \s(\xi) = - \int_{\sph} \Delta_0 f(\xi) g(\xi) \d \s(\xi).
\end{equation}
Together with \eqref{eq:LB-operator}, this identity implies that if $\{Y_\ell^n : 1 \le \ell \le \dim \CH_n^d\}$ is
an orthogonal basis of $\CH_n^d$, then
$$
 \int_{\sph} \nabla_0 Y_\ell^n(\xi) \cdot  \nabla_0 Y_{\ell'}^{n'}(\xi) \d \s(\xi)
     =\l_n \int_{\sph} Y_\ell^n(\xi) \cdot Y_{\ell'}^{n'}(\xi) \d \s(\xi) = \l_n \delta_{\ell,\ell'} \delta_{n,n'},
$$
where $\l_n = n(n+d-2)$, which implies that $\{Y_\ell^n : 1 \le \ell \le \dim \CH_n^d\}$ is also a family
of orthogonal polynomials for the Sobolev inner product defined by, for example,
\begin{equation}\label{eq:SOP-sph}
 \la f,g \ra_{\nabla} = \sum_{k=1}^s \mu_k
 \f 1 {\o_d} \int_{\sph} \nabla_0^k f(\xi) \cdot  \nabla_0^k g(\xi) \d \s(\xi) +  \f 1 {\o_d} \int_{\sph}   f(\xi) g(\xi) \d \s(\xi),
\end{equation}
where $\mu_k \ge 0$ and $s$ is a fixed positive integer, and
$$
   \nabla_0^{2m} := \Delta_0^m \quad\hbox{and} \quad \nabla_0^{2m+1} = \Delta_0^m \nabla_0.
$$
 In other words, the Sobolev orthogonal polynomials for $\la \cdot,\cdot\ra_\nabla$ on the unit sphere are trivially
 spherical harmonics themselves.
 
\section{Orthogonal polynomials on the conic surface}
\setcounter{equation}{0}

Orthogonal polynomials on the conic surface $\VV_0^{d+1}$ are studied in \cite{X20} for the inner product
defined by
\begin{equation}\label{ConeInnerProduct}
  \la f, g \ra_{\b,\g} = b_{\b,\g} \int_{\VV_0^{d+1}} f(x,t) g(x,t) \sw_{\b,\g}(t) \d \sm (x,t),
\end{equation}
where $\d \sm$ is the Lebesgue measure on the conic surface and the weight function $\sw_{\b,\g}$ is the
Jacobi weight function on $[0,1]$,
$$
\sw_{\b,\g}(t) = t^\b (1-t)^\g, \quad \b > -d, \quad \g > -1
$$
and $b_{\b,\g}$ is the normalization constant so that $\la 1, 1 \ra_{\b,\g} =1$, which is determined by
$$
 \int_{\VV_0^{d+1}} f(x,t) \d \sm(x,t) = \int_{0}^1 t^{d-1}  \int_{\sph} f(t \xi,t)\d \s(\xi) \d t,
$$
where $\d \s$ denotes the Lebesgue measure on the unit sphere $\sph$. Then,
$$
b_{\beta,\gamma}= \frac{1}{\omega_d} c_{\b+d-1,\g} \quad \hbox{with}\quad
c_{\b,\g} = \frac{\Gamma(\beta+\gamma+2)}{\Gamma(\beta+1)\Gamma(\gamma+1)} \quad \hbox{and}\quad
\omega_{d}=\frac{2\pi^{d/2}}{\Gamma(d/2)},
$$
where $\omega_{d}$ is the surface area of $\sph$. The inner product is well defined
for the space $\RR[x,t] / \la t^2 - \|x\|^2\ra$ of polynomials modulus the ideal generated by $t^2 -\|x\|^2$.

For $n \in \NN_0$, let $\CV_n^d(\sw_{\b,\g})$ denote the space of orthogonal polynomials of degree $n$. Since
$\VV_0^{d+1}$ is a quadratic surface, so the dimension of the space is the same as that of $\CH_n^{d+1}$.
Thus, $\dim \CV_0(\VV_0^{d+1},\sw_{\b,\g}) =1$ and
$$
   \dim \CV_n^d(\sw_{\b,\g})  = \binom{n+d-1}{n}+\binom{n+d-2}{n-1},\quad n=1,2,3,\ldots.
$$
An orthogonal basis for $\CV_n^d(\sw_{\b,\g})$ is given in \cite{X20} in terms of the Jacobi polynomials and
spherical harmonics. Let $\{Y_\ell^m: 1 \le \ell \le \dim \CH_m^d\}$ be an orthonormal basis of $\CH_m^d$.
Define the polynomials, called the Jacobi polynomials on the conic surface, by
\begin{equation} \label{eq:sfOPbasis}
  \sS_{m, \ell}^{n,(\b,\g)} (x,t) := P_{n-m}^{(2m + \b + d-1,\g)} (1-2t) Y_\ell^m (x),
\end{equation}
where, for $(x,t) \in \VV_0^{d+1}$, $Y_\ell^m(x)$ is a solid harmonic in $\CH_m^{d,0}$ and
$Y_\ell^m (t\xi) = t^m Y_\ell^m(\xi)$. Then
$\{ \sS_{m, \ell}^{n,(\b,\g)}: 0 \le m \le n, \,\, 1 \le \ell \le \dim \CH_m(\sph)\}$ is an orthogonal basis of
$\CV_n^d(\sw_{\b,\g})$, which satisfies
$$
 b_{\b,\g} \int_{\VV_0^{d+1}}  \sS_{m, \ell}^{n,(\b,\g)}  (x,t) \sS_{m', \ell'}^{n',(\b,\g)} (x,t)
     \sw_{\b,\g}(t)  \d \sm (x,t) =  h_{n,m}^{\b,\g}\delta_{n,n'}\delta_{m,m'}\delta_{\ell,\ell'},
$$
where $h_{n,m}^{\b,\g}$ is the square of the $L^2(\VV_0^{d+1},\sw_{\b,\g})$ norm of $\sS_{m, \ell}^n$ and
\begin{equation}\label{eq:Snorm}
   h_{n,m}^{\b,\g} = \frac{(\b+d)_{2m}}{(\b+\g+d+1)_{2m}} h_{n-m}^{(2m+\b+d-1,\g)}
\end{equation}
with $h_{n-m}^{(2m+\b+d-1,\g)}$ defined in \eqref{JacobiNorm}.

Of particular interest is the case $\b = -1$, for which the space $\CV_n^d(\sw_{\b,\g})$ is an eigenspace
of a second order linear differential operator. Parametrizing the space $\VV_0^{d+1}$ by $(x,t) = (t\xi, t)$
and let $\Delta_0^{(\xi)}$ be the Laplace-Beltrami operator on the unit sphere in the $\xi$ variable, which
is the restriction of the Laplace operator $\Delta$ on $\sph$. For $\g > -1$, define
\begin{equation} \label{eq:diff-op}
   \CD_\g =  t(1-t)\frac{\d^2}{\d t^2} + \big( d-1 - (d+\g)t \big) \frac{\d}{\d t}+ t^{-1} \Delta_0^{(\xi)}.
\end{equation}

\begin{thm}\label{thm:Jacobi-DE-V0}
Let $d\ge 2$ and $\g > -1$. The orthogonal polynomials in $\CV_n(\VV_0^{d+1}, \sw_{-1,\g})$ are eigenfunctions
of $\CD_{\g}$; more precisely,
\begin{equation}\label{eq:eigen-eqn}
    \CD_{\g} u =  -n (n+\g+d-1) u, \qquad \forall u \in \CV_n(\VV_0^{d+1}, \sw_{-1,\g}).
\end{equation}
\end{thm}

The differential operator $\CD_{\g}$ plays an important role in the study of the Fourier orthogonal series, 
and the best approximation by polynomials on the conic surface \cite{X21}. For $f \in L^2(\VV_0^{d+1}, \sw_{\b,\g})$,
the Fourier orthogonal series of $f$ is defined by
$$
  f = \sum_{n=0}^\infty \proj_n^{\b,\g} f,
$$
where $\proj_n^{\b,\g}: L^2(\VV_0^{d+1}, \sw_{\b,\g}) \rightarrow \CV_n^d(\sw_{\b,\g})$ is the orthogonal
projection operator, which satisfies, using the orthogonal basis given above,
$$
  \proj_n^{\b,\g} f = \sum_{m=0}^n \sum_{\ell=1}^{\dim \CH_m(\sph)}\hat f_{m,\ell}^{n, (\b,\g)} \sS_{m,\ell}^{n,(\b,\g)},
  \quad \hbox{where}\quad \hat f_{m,\ell}^{n, (\b,\g)} := \frac{\la f,  \sS_{m, \ell}^{n,(\b,\g)}\ra_{\b,\g}}{h_{m,n}^{\b,\g}}.
$$

In the rest of this section, we consider a property of the derivative of the projection operator. We need to be careful with
derivative for functions on the conic surface. A function on $\VV_0^{d+1}$ can be written as $f(x,t) = f(t \xi, t)$ with $\xi \in \sph$, 
so that $\frac{\d}{\d t}$ acts on both $x$ and $t$ variables, and it follows that
\begin{equation} \label{eq:partial_t}
  \frac{\d}{\d t} f(x,t)  =  \la \xi, \nabla_x\ra f(t\xi,t) + \partial_{d+1} f(t\xi,t) = \frac{1}{t} \la x, \nabla_x\ra  f(x,t) 
    +  \partial_{d+1} f(x,t)
\end{equation}
where $\partial_{d+1}$ denotes the partial derivative with respect to the $d+1$ variable of $f$. For our study, it is 
essential to consider the differential operator, acting on $f(x,t)$ on $\VV_0^{d+1}$, defined by
\begin{equation} \label{eq:fD}
   \fD =  \frac{\d}{\d t} - \frac{1}{t}  \la x, \nabla_x\ra, 
\end{equation}
which will be used throughout the rest of the paper. Geometrically, $\fD$ is the derivative in the normal direction to 
the cone $\VV_0^{d+1}$. By \eqref{eq:partial_t}, it follows immediately that
\begin{equation} \label{eq:fD0}
  \fD f(x,t) = \partial_{d+1} f(x,t). 
\end{equation}
To be more specific and further clarify the role of $\fD$, we state the following lemma. 

\begin{lem} \label{lem:fD}
For $(x,t) \in \VV_0^{d+1})$, let $p(t)$ and $q(x)$ be differentiable functions. Then 
\begin{align} \label{eq:fD1}
     \fD \left[p(t) q(x)\right] = p'(t) q(x). 
\end{align}
In particular, let $ \tau_{m,n} = - (n+m + \b+ \g + d)$, then
\begin{align} \label{eq:fD_basis}
     \fD \sS_{m,\ell}^{n,(\b,\g)}(x,t) & = \frac{\d}{\d t} \left[P_{n-m}^{(2m + \b + d-1,\g)} (1-2t)\right] Y_\ell^m(x) \\ 
       & =  \tau_{n,m} \sS_{m,\ell}^{n-1,(\b+1,\g+1)}(x,t), \qquad 0 \le m \le n-1,  \notag
\end{align}
and, for $m = n$, $\fD \sS_{n,\ell}^{n,(\b,\g)}(x,t) = 0$. 
\end{lem}

\begin{proof}
The identity \eqref{eq:fD1} is an immediate consequence of \eqref{eq:fD0}. Moreover, \eqref{eq:fD_basis} follows from the well-known 
formula for the derivative of the Jacobi polynomials, as
\begin{align*}
 \frac{\d}{\d t} \left[ P_{n-m}^{(2m + \b + d-1,\g)} (1-2t)\right] Y_\ell^m (\xi) 
   & = \tau_{n,m} P_{n-m-1}^{(2m + \b + d,\g+1)} (1-2t) Y_\ell^m (x) \\& = \tau_{n,m} \sS_{m,\ell}^{n-1,(\b+1,\g+1)}(x,t),
\end{align*}
which also holds for $m=n$, since the above derivative is zero when $n-m=0$.
\end{proof}

\begin{thm} \label{thm:partialP}
For $\b > -d$ and $\g > -1$, let $f$ be a differentiable function such that $f\in L^2(\VV_0^{d+1}, \sw_{\b,\g})$
and $\fD f \in  L^2(\VV_0^{d+1}, \sw_{\b+1,\g+1})$. Then, for $n =0,1,2,\ldots$,
$$
   \fD \proj_n^{\b,\g} f(x,t) = \proj_{n-1}^{\b+1,\g+1} \partial_t f (x,t).
$$
\end{thm}

\begin{proof}
By Lemma \ref{lem:fD}, $\fD \proj_n^{\b,\g} f \in \CV_{n-1}^d(\sw_{\b+1,\g+1})$. Consequently, it follows that
\begin{align*}
& \left \langle  \fD f, \sS_{m,\ell}^{n-1,(\b+1,\g+1)}\right \rangle_{\b+1,\g+1}  =
  \left \langle \fD \proj_n^{\b,\g} f,  \sS_{m,\ell}^{n-1,(\b+1,\g+1)} \right \rangle_{\b+1,\g+1}  \\
  & \qquad\qquad = \sum_{k=0}^{n-1} \sum_{\nu} \hat f_{k,\nu}^{n,(\b,\g)}  \tau_{n,k}  \left \langle\sS_{k,\nu}^{n-1,(\b+1,\g+1)},
        \sS_{m,\ell}^{n-1,(\b+1,\g+1)} \right \rangle_{\b+1,\g+1}  \\
  &  \qquad\qquad = \tau_{n,m}  \hat f_{m,\ell}^{n,(\b,\g)} h_{n-1,m}^{\b+1,\g+1},
\end{align*}
which implies immediately that
$$
    \wh{\fD f}_{m,\ell}^{n-1, (\b+1,\g+1)} = \tau_{n,m} \hat f_{m,\ell}^{n,(\b,\g)}, \quad 0 \le m \le n-1.
$$
Consequently, we obtain
\begin{align*}
  \fD \proj_n^{\b,\g} f (x,t)
      \, & = \sum_{m=0}^{n-1} \sum_\ell \hat f_{m,\ell}^{n,(\b,\g)}\tau_{n,m} \sS_{m,\ell}^{n-1,(\b+1,\g+1)}(x,t) \\
       & = \sum_{m=0}^{n-1} \sum_\ell \wh{\fD f}_{m,\ell}^{n-1, (\b+1,\g+1)} \sS_{m,\ell}^{n-1,(\b+1,\g+1)}(x,t) \\
       & = \proj_{n-1}^{\b+1,\g+1} \fD f (x,t),
\end{align*}
where in the first step we used again $\fD_t \sS_{n,\ell}^n(x,t) = 0$.
\end{proof}

For $1 \le p \le \infty$, let $\|f\|_{p, \sw_{\b,\g}}$ denote the norm of the space $L^p(\VV_0^{d+1}, \sw_{\b,\g})$, and
we adopt the convention that the space is $C(\VV_0^{d+1})$ with the norm taken as the uniform norm when
$p = \infty$. Let $\Pi_n(\VV_0^{d+1})$ denote the space of polynomials of degree $n$ restricted on the
$\VV_0^{d+1}$. For $f \in L^p(\VV_0^{d+1}, \sw_{\b,\g})$, the quantity
$$
     E_n(f)_{p, \sw_{\b,\g}} := \inf_{P \in \Pi_n(\VV_0^d)} \| f - P \|_{p, \sw_{\b,\g}}
$$
is the error of the best approximation by polynomials of degree at most $n$ in the norm of $L^p(\VV_0^{d+1}, \sw_{\b,\g})$.
We call $\eta \in C^\infty$ an admissible cut-off function if it is supported on $[0,2]$ and satisfies
$\eta(t) = 1$ if $0 \le t \le 1$. Let $\eta$ be such a function; we define
\begin{equation} \label{eq:Qn-eta}
  Q_{n,\eta}^{(\b,\g)} f = \sum_{k=0}^{2n} \eta\left(\f k n \right) \proj_k^{\b,\g} f.
\end{equation}
Then it is known \cite{X21} that $Q_{n,\eta}^{(\b,\g)}$ is a bounded operator in $L^p(\VV_0^{d+1}, \sw_{\b,\g})$ and it is
a polynomial of near best approximation in the sense that
\begin{equation} \label{eq:Qn-etaEstimate}
    \left \| f -  Q_{n,\eta}^{(\b,\g)} f\right \|_{p, \sw_{\b,\g}} \le c E_n (f)_{p,\sw_{\b,\g}},
\end{equation}
where $c$ is a constant that depends only on $\eta$, $p$, $\b$ and $\g$. In particular, we obtain the following as
a corollary of Theorem \ref{thm:partialP}.

\begin{cor} \label{cor:approxQ}
Let $\b > -d$ and $\g > -1$. Let $r$ be a postive integer and let $f \in C^r(\VV_0^{d+1})$ such that
 $\fD^k f \in  L^p(\VV_0^{d+1}, \sw_{\b+k,\g+k})$ for $0 \le k \le r$.  Then, for $1 \le p \le \infty$,
$$
  \left \| \fD^k f - \fD^k Q_{n,\eta}^{(\b,\g)}f\right \|_{p, \sw_{\b+k,\g+k} } \le c_k
      E_{n-k} (\fD^k f)_{p, \sw_{\b+k,\g+k} }, \qquad 0 \le k \le r.
$$
\end{cor}

\begin{proof}
This follows immediately from
\begin{align*}
   \fD^k Q_{n,\eta}^{(\b,\g)} f & =   \sum_{j=0}^{2n} \eta\left(\f j n \right)  \fD^j\proj_j^{\b,\g} f
       =   \sum_{j=k}^{2n} \eta\left(\f j n \right) \proj_{j-k}^{\b+k,\g+k} f \\
      & =   \sum_{j=0}^{2n-k} \eta\left(\f {j+k}n \right) \proj_{j}^{\b+k,\g+k}  \fD^k f
         = \sQ_{n,\eta}^{(\b+k,\g+k)}  \fD^k  f,
\end{align*}
where we define
\begin{equation} \label{eq:wtQ}
  \sQ_{ n,\eta}^{(\b,\g)} g =  \sum_{j=0}^{2n-k} \eta\left(\f {j+k}n \right) \proj_{j}^{\b,\g} g, \quad 0 \le k < 2n.
\end{equation}
For fixed $k \le r$ independent of $n$, the function $\eta\left(\f {j+k}n \right)$ plays essentially the same role
as $\eta\left(\f {j}n \right)$, so that \eqref{eq:Qn-etaEstimate} holds with $\sQ_{n,\eta}^{(\b,\g)} \fD^k f$ in place
of $Q_{n,\eta}^{(\b,\g)}f$, form which the proof follows readily.
\end{proof}

\section{Sobolev orthogonal polynomials on the conic surface}
\setcounter{equation}{0}

For the Sobolev inner product, we again use the differential operator $\fD$ defined in \eqref{eq:fD}. 
For $s \in \NN$, let us define 
$$
W_p^s(\VV_0^{d+1}, \sw_{\b+s,0})= \left \{f\in C(\VV_0^{d+1}): \fD^s f\in L^p(\VV_0^{d+1},\sw_{\b+s,0})\right \},
$$
where $1 \le p \le \infty$ and the space is taken as the $C^s(\VV_0^{d+1})$ if $p = \infty$.

Let $s$ be a positive integer and $\b > -s-d$. Let $\l_1,\ldots,\l_{r-1}$ be fixed positive numbers. 
We consider the Sobolev inner product defined by
\begin{align} \label{eq:ipd-s}
  \la f,g\ra_{\b, -s} =\,& \frac{1}{\o_d}
       \int_{\VV_0^{d+1}} \fD^s f (x,t) \fD^s g(x,t) t^{\b+s} \d \sm(x,t)\\\
           &  + \sum_{k=0}^{s-1} \frac{\l_k}{\o_d} \int_{\SS^{d-1}}
              \fD^k f(\xi,1) \fD^k g(\xi,1) \d \s (\xi), \notag
\end{align}
which is evidently an inner product on the space $W_2^s(\VV_0^{d+1}, \sw_{\b+s,0})$. We denote by
$\CV_n^d(\sw_{\b,-s})$ the space of orthogonal polynomials of degree $n$ with respect to this inner product.
Our first task is to find an orthogonal basis for this space.

Recall the modified Jacobi polynomial $J_n^{(\a,-s)}$ defined in \eqref{eq:CJ}. Let
$\{Y_\ell^m: 1 \le \ell \le \dim \CH_m^d\}$ be an orthonormal basis of $\CH_m^d$. For $1 \le \ell \le \dim \CH_m^d$,
$0\le m \le n$, we define
\begin{equation} \label{eq:SOP-s}
    \sS_{m,\ell}^{n, (\b,-s)} (x,t) := J_{n-m}^{(2m+\b+d-1,-s)}(1-2 t) Y_\ell^m(x).
\end{equation}

\begin{thm} \label{thm:SOP-cone}
Let $s \in \NN$ and $\b > -d-s$. The polynomials $\sS_{m,\ell}^{n, (\b,-s)}$, $1 \le \ell \le \dim \CH_m^d$
and $0\le m \le n$ consist of a basis of $\CV_n^d(\sw_{\b,-s})$. Moreover, for all $\ell$, $\sh_{m,n}^{(\b,-s)}  =
\langle \sS_{m,\ell}^{n, (\b,-s)}, \sS_{m,\ell}^{n, (\b,-s)} \rangle_{\b,-s}$ satisfies
$$
 \sh_{m,n}^{(\b,-s)}   = \begin{cases}2^{2n-2m} \l_{n-m} &  0 \le n -m \le s-1, \\
      2^{s - \b-2m-d}\, \wh h_{n-m-s}^{(s+2m+\b+d-1,0)} & n-m \ge s. \end{cases}
$$
\end{thm}

\begin{proof}
Let $q_{n-m}(t) = J_{n-m}^{(2m+\b+d-1, -s)}(t)$. Changing variable $x = t \xi$ and using \eqref{eq:fD_basis}, we obtain
\begin{align*}
 \left \langle \sS_{m,\ell}^{n,(\b,-s)}, \sS_{m',\ell'}^{n',(\b,-s)} \right \rangle_{\b,-s}
  = \, & 2^{ 2s}  \int_0^1 t^{d-1} q_{n-m}^{(s)} (1-2t) q_{n'-m'}^{(s)}(1-2t) t^{\b+s+m+m'} \d t \\
 &\qquad \times  \frac{1}{\o_d} \int_{\SS^{d-1}}  Y_\ell^m (\xi)Y_{\ell'}^{m'} (\xi) \d \s (\xi) \\
 & + \sum_{k=0}^{s-1} \l_k 2^{2k} q_{n-m}^{(k)}(-1) q_{n'-m'}^{(k)}(-1)  \\
 &\qquad \times  \frac{1}{\o_d} \int_{\SS^{d-1}}    Y_\ell^m (\xi)Y_{\ell'}^{m'} (\xi) \d \s (\xi).
\end{align*}
Using the orthogonality of $Y_\ell^m$ and changing the variable $t \mapsto (1-t)/2$ in the first integral, we
see that the expression containing $q_{n-m}$ can be written in terms of the Sobolev inner product
$[\cdot,\cdot]_{\a,\b}^{-s}$ defined in \eqref{eq:SOPd=1}. More precisely, we obtain
$$
\left \langle \sS_{m,\ell}^{n,(\b,-s)},  \sS_{m',\ell'}^{n',(\b,-s)} \right \rangle_{\b,-s} =
   2^{2 s - \a-1} [q_{n-m}, q_{n'-m}]_{\a,0}^{-s} \delta_{m,m'} \delta_{\ell,\ell'},
$$
where $\a = s+\b+2m+d-1$ and the inner product $[\cdot, \cdot]_{\a,0}^{-s}$ is defined as in \eqref{eq:SOPd=1}
but with $\mu_k = 2^{2k} \l_k/2^{2 s-\a-1}$. As shown in Theorem \ref{thm:SOP-1d}, the polynomials $J_n^{(\a-s,-s)}$
are orthogonal with respect to this inner product, regardless of the values of $\mu_k$. This
proves the orthogonality of $\sS_{m,\ell}^{n,(\b,-s)}$. Moreover,  the norm $\sh_{m,n}^{(\b,-s)}$ follows
from the norm given in Theorem \ref{thm:SOP-1d}.
\end{proof}

Recall that the constant $A_n^{(\a,\b)}$ is defined in \eqref{eq:JacobiJ}.

\begin{cor}\label{cor:der_sS}
For $0 \le m \le n-s$,
\begin{equation} \label{eq:D_sS}
 \fD^s \sS_{m,\ell}^{n, (\b,-s)}(x,t) = (-2)^s A_{n-s-m}^{(s+2m+\b+d-1,0)} S_{m,\ell}^{n-s, (\b+s,0)}(x,t),
\end{equation}
and it is equal to zero if $m > n-s$. Moreover, for $1 \le k \le s-1$ and $\xi \in \sph$,
\begin{equation} \label{eq:D_sS2}
    \fD^k  \sS_{m,\ell}^{n, (\b,-s)}(\xi,1) = \begin{cases}
         (-2)^k Y_\ell^m(\xi) \delta_{k,n-m}, & m > n-s \\
         0, & m \le n-s.
     \end{cases}
\end{equation}
\end{cor}

\begin{proof}
If $m \le n-s$, by \eqref{eq:fD_basis} and the identity \eqref{eq:CJ1a},
\begin{align*}
\fD^s \sS_{m,\ell}^{n, (\b,-s)}(x,t)  \,&  =
     \frac{\d^s}{\d t^s}\left[J_{n-m}^{(2m+\b+d-1,-s)}(1-2t)\right]Y_\ell^m(x) \\
   & = (-2)^{s} \wh P_{n-m-s}^{(s+2m+\b+d-1,0)}(1-2t)Y_\ell^m(x),
\end{align*}
from which \eqref{eq:D_sS} follows from \eqref{eq:JacobiJ} and the definition of $\sS_{\ell,m}^{n,(\b+s,0)}$.
Moreover, since $n-m \ge s$, for $1 \le k \le s-1$ it follows immediately by \eqref{eq:CJ1b} that \eqref{eq:D_sS2}
holds. If $m > n-s$, then the derivative in the identity \eqref{eq:D_sS} is evidently zero since the Jacobi polynomial
in $\sS_{m,\ell}^{n, (\b,-s)}$ is of degree $n-m <s$. Moreover, for $1 \le k \le s-1$, \eqref{eq:D_sS2} follows
from \eqref{eq:CJ1b}.
\end{proof}

Our notation for the Sobolev orthogonal polynomials $\sS_{m,\ell}^{n,(\b,-s)}$ is the same as the one
for the ordinary orthogonal polynomials $\sS_{m,\ell}^{n,(\b,\g)}$ with $\g = -s$. This is intentional as can be
seen in \eqref{eq:Jn(a,-s)=Pn}, which however only works for $n \ge s$ or $\sS_{m,\ell}^{n,(\b,-s)}$ for
$n-m \ge s$. We can, moreover, give an explicit expression of the Sobolev orthogonal polynomials by
using the identity \eqref{eq:Jn(a,-s)}.

\begin{cor}
For $0 \le \ell \le \dim \CH_m^d$ and $0 \le m \le n$, the polynomials $S_{m,\ell}^{n,(\b,-s)}$ satisfy
\begin{equation}\label{eq:sS-explicit}
   \sS_{m,\ell}^{n,(\b,-s)}(x,t) = \begin{cases} b_{m,n}^s  (1- t)^s S_{m,\ell}^{n-s,(\b,s)}(x,t) & n-m \ge s, \\
           \frac{2^{n-m}}{(n-m)!}  (1-t)^{n-m} Y_\ell^m(x), &  0 \le n-m\le s-1,
\end{cases}
\end{equation}
where $b_{m,n}^s =   (-1)^s 2^s A_{n-s-m}^{(2m+\b+d-1,s)}/{(-n+m)_s}$. In particular, the space $\CV_n^d(\sw_{\b,-s})$
satisfies a decomposition
\begin{equation}\label{eq:decompVnd}
  \CV_n^d(\sw_{\b,-s}) = \bigoplus_{j=0}^{s-1} (1-t)^j \CH_{n-j}^{d,0} \bigoplus (1-t)^s \CV_{n-s}^d(\sw_{\b,s}).
\end{equation}
\end{cor}

\begin{proof}
This follows from the identity \eqref{eq:Jacobi_n<0}, which shows, together with \eqref{eq:Jn(a,-s)=Pn}, that
$$
  J_{n}^{(\a,-s)}(1-2t) = \frac{(-1)^s 2^s}{(-n)_s} A_{n-s}^{(\a,s)}(1-t)^s P_{n-s}^{(\a,s)}(1-2t)
$$
for $n\ge s$, which implies, with $n$ replaced by $n-m$, the identity \eqref{eq:sS-explicit} when $n - m \ge s$.
For $n -m \le s-1$, \eqref{eq:sS-explicit} follows immediately from the definition of $J_n^{(\a,-s)}$ in
\eqref{eq:CJ}. Since $\{Y_\ell^m: 1 \le \ell \le m\}$ is a basis of $\CH_n^{d,0}$, the decomposition
\eqref{eq:decompVnd} is an immediate consequence of \eqref{eq:sS-explicit}.
\end{proof}

With the expression \eqref{eq:sS-explicit} for $\sS_{m,\ell}^{n, (\b,-s)}$, we could bypass the integral
definition of $J_n^{(\a,-s)}$ in \eqref{eq:CJ}. However, the integral definition is more convenient for studying
the Fourier orthogonal series, as shown below.

For $f \in W_2^s(\VV_0^{d+1}, \sw_{\b+s,0})$, the Fourier orthogonal
series of $f$ is defined by
$$
  f = \sum_{n=0}^\infty \proj_n^{\b,-s} f,
$$
where $\proj_n^{\b,-s}: W_2^s(\VV_0^{d+1}, \sw_{\b+1,0}) \rightarrow \CV_n^d(\sw_{\b,-s})$ is the orthogonal
projection operator, which satisfies, using the orthogonal basis given above,
$$
  \proj_n^{\b,-s} f = \sum_{m=0}^n \sum_{\ell}\hat f_{m,\ell}^{n, (\b,-s)} \sS_{m,\ell}^{n,(\b,-s)},
  \quad \hbox{where}\quad \hat f_{m,\ell}^{n, (\b,-s)} := \frac{\left \langle f,  \sS_{m, \ell}^{n,(\b,-s)}\right \rangle_{\b,-s}}
      {\sh_{m,n}^{(\b,-s)}}.
$$
Using the basis in Theorem \ref{thm:SOP-cone} and its corollary, we can derive an integral
representation for the projection operator $\proj_n^{\b,-s} f$.

\begin{thm} \label{thm:proj-s=int}
Let $s$ be a positive integer. For $\b > -d-s$, let $f$ be a differentiable function such that
$\fD^s f \in L^2(\VV_0^{d+1}, \sw_{\b+s,0})$. Then, for $n =0,1,2,\ldots$ and $(x,t) \in \VV_0^{d+1}$,
\begin{align}\label{eq:proj-s=int}
   \proj_n^{\b,-s} f(x,t) = \,& \sum_{m=0}^{s-1} \frac{(t-1)^m}{m!} \proj_{n-m}^\sph  \left[\fD^m f(\cdot,1)\right](x) \\
        & + (-1)^s \int_{t}^{1} \frac{(v-t)^{s-1}}{(s-1)!} \proj_{n-s}^{\b+s,0} \left(\fD^s f\right) (x, v) \d v. \notag
\end{align}
\end{thm}

\begin{proof}
Let $\a = s+\b+d-1$. For $0 \le m \le n-s$, it follows from \eqref{eq:D_sS} that
\begin{align*}
\left \langle f,  \sS_{m, \ell}^{n,(\b,-s)}\right \rangle_{\b,-s} \,&  =(-2)^s A_{n-s-m}^{(2m+\a,0)}
   \frac{1}{\o_d} \int_{\VV_0^{d+1}} \fD^s f(x,t) \sS_{m,\ell}^{n-s,(\b+s,0)}(x,t) t^{\b+s} \d \sm(x,t)\\
         & = (-2)^s A_{n-s-m}^{(2m+\a,0)} \frac{1}{c_{\a,0}}
           \left \langle \fD^s f,  \sS_{m, \ell}^{n-s,(\b+s,0)}\right \rangle_{\b+s,0}
\end{align*}
and the norm of $\sS_{m,\ell}^{n,(\b,-s)}$ satisfies
\begin{align*}
  \sh_{m,n}^{(\b,-s)} = 2^{s - \b-2m-d} \wh h_{n-m-s}^{(\a+2m,0)} \, & =
   2^{2s}  \frac{1}{c_{2m+\a,0}} \left[A_{n-m-s}^{(2m+\a,0)}\right]^2 h_{n-s-m}^{(2m+\a,0)} \\
  & = \frac{1}{c_{\a,0}} 2^{2s} \left[A_{n-m-s}^{(2m+\a,0)}\right]^2 \frac{(\a+1)_{2m}}{(\a+2)_{2m}}h_{n-s-m}^{(2m+\a,0)}\\
  & =  \frac{1}{c_{\a,0}} 2^{2s} \left[A_{n-m-s}^{(2m+\a,0)}\right]^2 h_{n-s, m}^{\b+s,0}.
\end{align*}
Consequently, it follows that for $0 \le m \le n-s$,
\begin{equation}\label{eq:FourierCoeffs}
 \hat f_{m,\ell}^{n, (\b,-s)} = \frac{ \left \langle \fD^s f,  \sS_{m, \ell}^{n-s,(\b+s,0)}\right \rangle_{\b+s,0}}
       { (-2)^{s} A_{n-m-s}^{(2m+\a,0)} h_{n-s, m}^{\b+s,0}}
        = \frac{ \wh{\fD^s f}_{m,\ell}^{n-s, (\b+s,0)}}{(-2)^{s} A_{n-m-s}^{(2m+\a,0)}}.
\end{equation}
Now, for $n-s < m \le n$, it follows by Corollary \ref{cor:der_sS} that
\begin{align*}
\left \langle f, \sS_{m, \ell}^{n,(\b,-s)}\right \rangle_{\b,-s}
\,& = \sum_{k=0}^{s-1} \frac{\l_k}{\o_d} \int_{\sph} \fD^k f(\xi,1) \fD^k  \sS_{m, \ell}^{n,(\b,-s)}(\xi,1)\d \s(\xi) \\
 & = (-2)^{n-m} \frac{\l_{n-m}}{\o_d} \int_{\sph}\fD^{n-m} f(\xi,1) Y_\ell^{m} (\xi)\d \s(\xi),
\end{align*}
so that, using the norm of $\sh_{n,m}^{(\b,-s)}$ in Theorem \ref{thm:SOP-cone},
$$
\hat f_{m,\ell}^{n, (\b,-s)} = \frac{1}{(-2)^{n-m}} \frac{1}{\o_d} \int_{\sph} \fD^{n-m} f(\xi,1) Y_\ell^{m} (\xi)\d \s(\xi),
   \quad 0\le n-m \le s-1.
$$
These allow us to expand $\proj_n^{\b,-s} f$, for which we split the sum over $\sum_{m=0}^n$ into two parts. We make 
a change of variable for the first sum $\sum_{m=n-s}^n B_m = \sum_{m=0}^{s-1} B_{n-m}$ and write the polynomial 
$\sS_{n-m,\ell}^{n, (\b,-s)}$ using the second case of \eqref{eq:sS-explicit}, whereas for the second sum $\sum_{m=0}^{n-s} B_m$
we use the integal form in \eqref{eq:CJ} in $\sS_{n-m,\ell}^{n, (\b,-s)}$. We then obtain, using $\a = s+\b+d-1$ again, 
\begin{align*}
 & \proj_n^{\b,-s} f(x,t) \, = \sum_{m=0}^{s-1} \frac{(t-1)^m}{m!}
      \sum_{\ell=1}^{\dim \CH_{n-m}^d} \frac{1}{\o_d} \int_{\sph}
         \fD^{m} f(\eta,1) Y_\ell^{n-m} (\eta)\d \s(\eta)\,Y_\ell^{n-m} (x) \\
   & \qquad\quad
    + \sum_{m = 0}^{n-s} \sum_{\ell=1}^{\dim \CH_{m}^d}
       \frac{ \wh{\fD^s f}_{m,\ell}^{n-s, (\b+s,0)}}{(-2)^{s} A_{n-m-s}^{(2m+\a,0)}}
        \int_{-1}^{1-2t} \frac{(1-2t -u)^{s-1}}{(s-1)!} \wh P_{n-s}^{(\a+2m,0)}(u) \d u \, Y_\ell^m(x).
 \end{align*}
The inner sum of the first term on the right-hand side is
$$
t^{n-m} \proj_{n-m}^\sph [\fD^m f(\cdot,1)](\xi) = \proj_{n-m}^\sph [\fD^m f(\cdot,1)](x),
$$
since the function $\xi \mapsto \fD^m f(\xi,1)$ on the unit sphere and $\proj_{n-m}^\sph$ are homogeneous of degree
$n-m$. Changing variable $u = 1-2v$ and using \eqref{eq:JacobiJ} in the integral, the second term on the right-hand side becomes,
\begin{align*}
  \sum_{m = 0}^{n-s} \sum_{\ell} & \wh{\fD^s f}_{m,\ell}^{n-s, (\b+s,0)}  (-1)^s
        \int_{t}^{1} \frac{(v-t)^{s-1}}{(s-1)!} P_{n-s}^{(\a+2m,0)}(1-2v) \d v \, Y_\ell^m(x) \\
   & = (-1)^s  \int_{t}^{1}  \frac{(v-t)^{s-1}}{(s-1)!}  \sum_{m = 0}^{n-s} \sum_{\ell}
     \wh{\fD^s f}_{m,\ell}^{n-s, (\b+s,0)}  \sS_{m,\ell}^{n-s, (\b+s,0)} (x,v) \d v\\
  &  =  (-1)^s \int_{t}^{1} \frac{(v-t)^{s-1}}{(s-1)!} \proj_{n-s}^{\b+s,0}\fD^s f (x, v) \d v.
\end{align*}
Putting the two terms together completes the proof.
\end{proof}

\begin{cor}\label{cor:proj-partial}
Let $s$ be a positive integer. For $\b > -d-s$, let $f$ be a differentiable function such that
$\fD^s f \in L^2(\VV_0^{d+1}, \sw_{\b+s,0})$. Then, for $n  \ge s$,
$$
  \fD^s \proj_n^{\b,-s} f(x,t) = \proj_{n-s}^{\b+s,0}\left(\fD^s f\right) (x,t).
$$
\end{cor}

\begin{proof}
Taking the derivative $\fD^s$ of \eqref{eq:proj-s=int}, it follows from \eqref{eq:fD1} that the first term in the right-hand
side is zero, while the second term gives the stated identity.
\end{proof}

Another consequence of Theorem \ref{thm:proj-s=int} is an expression for the error of approximation.
Let $\eta \in C^\infty(\RR_+)$ be an admissible cut-off function. We denote by $\sQ_{n,\eta}^{(\beta, -s)}$
the near best approximation operator defined in \eqref{eq:Qn-eta} but with $\g = -s$. Furthermore, let
$$
\sQ_{n-m}^\sph f (\xi)= \sum_{k=0}^{2(n-m)} \eta\left(\frac{k}{n}\right) \proj_k^\sph f(\xi), \quad 0 \le m \le n,
$$
be a near-best approximation operator of degree $n-m$ on the unit sphere. For $f\in L^p(\sph)$, let
$$
  E_n^\sph(f)_p = \inf_{Y \in \Pi_n(\sph)} \| f - Y\|_{L^p(\sph)}, \qquad 1 \le p \le \infty
$$
be the error of best approximation by polynomials on the unit sphere, where the norm is the uniform norm for
$p = \infty$ and $\Pi_n(\sph)$ denotes the space of polynomials of degree at most $n$ restricted on the unit
sphere $\sph$. Then (cf. \cite[Theorem 2.6.3.]{DaiX})
$$
   \left  \|f - \sQ_{n}^\sph f \right \|_p \le c E_n^{\sph}(f)_p, \qquad 1 \le p \le \infty.
$$

\begin{thm} \label{thm:errorSOP}
For $f \in C^s(\VV_0^{d+1})$,
\begin{align*}
  f(x,t) - \sQ_{n,\eta}^{(\beta, -s)}f(x,t)  \,& = \sum_{m=0}^{s-1} \frac{(t-1)^m}{m!} \left[ \fD^m f(x,,1)
     - \sQ_{n-m,\eta}^{\sph} \left[\fD^m f(\cdot,1)\right](x) \right] \\
   & +  (-1)^s \int_{t}^{1} \frac{(v-t)^{s-1}}{(s-1)!} \left[ \fD^s f (x, v)-
      \sQ_{n-s,\eta}^{\b+s,0}\left(\fD^s f\right) (x, v) \right] \d v.
\end{align*}
\end{thm}

\begin{proof}
Multiplying \eqref{eq:proj-s=int} by $\eta\left( \frac{k}{n} \right)$ and summing up over $k$ gives an expression
of $\sQ_{n,\eta}^{(\b,-s)} f$. Furthermore, by \eqref{eq:fD0}, the Tylor expansion of $t\rightarrow f(x,t)$ with respect to the $t$
variable at $t =1$, together with its remainder formula, gives 
$$
  f(x,t) = \sum_{m=0}^{s-1} \frac{(t-1)^m}{m!}  \fD^m f(x,1) +
   (-1)^s \int_{t}^{1} \frac{(v-t)^{s-1}}{(s-1)!} \fD^s f (x, v) \d v.
$$
Together, the difference of the two identities gives the error formula for $f -\sQ_{n,\eta}^{(\b,-s)} f$ and completes
the proof.
\end{proof}

Taking the $k$-th order derivative with respect to the $t$ variable for $ 0 \le t \le s-1$, we obtain, by \eqref{eq:fD1},
\begin{align*}
 \fD^k \left(f  - \sQ_{n,\eta}^{(\beta, -s)}f \right)&(x,t)  = \sum_{m= k}^{s-1} \frac{(t-1)^{m-k}}{(m-k)!}
      \left [  \fD^m f(x,1)
     - \sQ_{n-m,\eta}^{\sph} [ \fD^m f(\cdot,1)](x) \right] \\
   & +  (-1)^{s-k} \int_{t}^{1} \frac{(v-t)^{s-k-1}}{(s-k-1)!} \left[  \fD^s f (x, v)-
      \sQ_{n-s,\eta}^{(\b+s,0)}\left( \fD^s f\right) (x, v) \right] \d v,
\end{align*}
moreover, taking the $\fD^s$ derivative in $(x,t)$ variable, we obtain
\begin{equation} \label{eq:partial-s}
     \fD^s \left(f  - \sQ_{n,\eta}^{(\beta, -s)}f \right)(x,t) =  \fD^s f(x,t)-\sQ_{n-s,\eta}^{(\b+s,0)}\left( \fD^s f\right) (x, t).
\end{equation}
In the above error formulas for $1\le k \le s-1$, the integral is over $v \in [t,1]$ while the variable is
$(x,v) = (t \xi, v)$. As it is, we cannot deduce the error estimate for $\fD^k \big(f  - \sQ_{n,\eta}^{(\beta, -s)}f \big)$ 
immediately from that of $f -  \sQ_{n-s,\eta}^{(\b+s,0)}\big(\fD^s f\big)$ over $\VV_0^{d+1}$ if $1 \le k \le s-1$,
although we can if $k =s$ by \eqref{eq:partial-s}. We state the latter case as a corollary.

\begin{cor}
Let $f \in W_p^s(\sw_{\b+s,0})$. Then, for $1 \le p \le \infty$,
$$
   \left \|\fD^s f - \fD^s \sQ_{n,\eta}^{(\beta, -s)}f \right\|_{p, \sw_{\b+s,0}}
    \le c E_n\left( \fD^s f \right)_{p, \sw_{\b+s,0}}.
$$
\end{cor}

We can derive another interesting property of the projection operator for the SOPs, which relies on
the property that the Jacobi polynomial $P_n^{(\a,-s)}$ is related to $P_{n-s}^{(\a,s)}$, as seen in
\cite[(4.22.2)]{Sz}. The following lemma shows that our extension $J_n^{\a,-s}$ satisfies the same
property.

\begin{thm}
If $f(x,t) = (1-t)^s g(x,t)$, then
$$
   \proj_n^{\b,-s} f = (1-t)^s \proj_{n-s}^{\b,s} g.
$$
\end{thm}

\begin{proof}
By the definition of the $\sS_{m,\ell}^{n,(\b,-s)}$, it follows from \eqref{eq:Jn(a,-s)} that
$$
 \sS_{m,\ell}^{n, (\b,-s)} (x,t)= \s_{m,n} (1-t)^s \sS_{m,\ell}^{n-s, (\b,s)}(x,t), \quad
 \s_{m,n} :=  \frac{(n-s-m)!}{(n-m)!} A_{n-m}^{(2m+\b+d-1,0)}.
$$
Expanding $g$ in terms of the Fourier orthogonal series in $L^2(\VV_0^{d+1}, \sw_{\b,1})$, we obtain
\begin{align*}
 f(x,t) = (1-t)^s g(x,t) \,& = (1-t)^s \sum_{n=0}^\infty \sum_{m=0}^n \sum_\ell \hat g_{m,\ell}^{n, (\b,s)}
    \sS_{m,\ell}^{n, (\b,s)}(x,t) \\
    & = \sum_{n=0}^\infty \sum_{m=0}^n \sum_\ell \hat g_{m,\ell}^{n, (\b,s)}
    \s_{n+s,m}^{-1} \sS_{m,\ell}^{n+s, (\b,-s)}(x,t).
\end{align*}
Applying the orthogonality in the Sobolev inner product, it follows immediately that
$$
 \left \langle f,  \sS_{m,\ell}^{n+s, (\b,-s)} \right \rangle_{\b,-s} =
  \s_{n+1,m}^{-1} \hat g_{m,\ell}^{n,(\b,s)} \left \langle
    \sS_{m,\ell}^{n+s, (\b,-s)},  \sS_{m,\ell}^{n+s, (\b,-s)} \right \rangle_{\b,-s}, \quad 0 \le m \le n;
$$
in particular, with $n$ replaced by $n-s$, we then obtain
\begin{equation} \label{eq:Fourirer-1}
      \hat f_{m,\ell}^{n,(\b,-s)} = \s_{n,m}^{-s} \hat g_{m,\ell}^{n-s,(\b,s)}, \quad 0 \le m \le n-s.
\end{equation}
For $n-s < m \le n$, the polynomial $J_{n-m}^{(2m+\b+d-1,-s)}(t)$ is a polynomial of degree at most $s-1$,
so that $\fD^s \sS_{n,\ell}^{n, (\b,-1)}(x,t) = 0$; moreover, $\fD^k f(\xi,1) =0$ for $0 \le k \le s-1$,
since $f$ contains a factor $(1-t)^s$. It follows that $\hat f_{n,\ell}^{n,(\b,-s)} = 0$ for $n -s < m \le n$ by the
definition of the Sobolev inner product. Consequently, by \eqref{eq:Jn(a,-s)} and \eqref{eq:Fourirer-1} we
obtain
\begin{align*}
    (1-t) \proj_{n-s}^{\b,s} g\, & = (1-t) \sum_{m=0}^{n-s} \sum_\ell \hat g_{m,\ell}^{n-1,(\b,s)} \sS_{m,\ell}^{n-s, (\b,s)} \\
                   & = \sum_{m=0}^{n} \sum_\ell \hat f_{m,\ell}^{n,(\b,-s)} \sS_{m,\ell}^{n, (\b,-s)} = \proj_{n}^{\b,-s}f.
\end{align*}
The proof is completed.
\end{proof}

\begin{cor}
Let $f \in L^p(\VV_0^{d+1},\sw_{\b,0})$ and $f(t) = (1-t) g(t)$. Then, for $1 \le p \le \infty$,
$$
   \left \| f - \sQ_{n,\eta}^{(\beta, -s)} f \right\|_{p, \sw_{\b,0}}
    \le c E_n\left(g \right)_{p, \sw_{\b,s}} \le c E_n\left(g \right)_{p, \sw_{\b,0}}
$$
\end{cor}

\begin{proof}
Since $f = (1-t)^s g$, we see that
$$
\sQ_{n,\eta}^{(\beta, -1)} f = \sum_{k=0}^{2n} \eta\left( \frac{k}{n}\right) (1-t)^s \proj_{k-1}^{\b,s} g
    = (1-t)^s \wt \sQ_{n,\eta}^{\beta,s} g,
$$
where $\wt \sQ_{n,\eta}^{(\beta,s)}$ is defined as in the proof of Corollary \ref{cor:approxQ}.
For $p \ge 1$, we then obtain
\begin{align*}
  \left \| f - \sQ_{n,\eta}^{(\beta, -s)} f \right\|_{p, \sw_{\b,0}}
  & =  \left \| (1-t)^s \left[ g - \wt \sQ_{n,\eta}^{(\beta, s)} g \right] \right\|_{p, \sw_{\b,0}}\\
  & \le \left \|g - \wt \sQ_{n,\eta}^{(\beta, s)} g \right\|_{p, \sw_{\b,s}} \le c E_n(g)_{p, \sw_{\b,s}},
\end{align*}
By the definition of $E_n(g)_{p, \sw}$, it follows immediately $E_n(g)_{p, \sw_{\b,s}} \le E_n(g)_{p, \sw_{\b,0}}$.
The proof is completed.
\end{proof}

\section{Partial differential equation for the Sobolev orthogonal polynomials}
\setcounter{equation}{0}

This section considers partial differential equations satisfied by the Sobolev orthogonal polynomials.
As we mentioned in Theorem \ref{thm:Jacobi-DE-V0} proved in \cite{X20}, the ordinary orthogonal polynomials
in $\CV_n(\VV_0^{d+1}, \sw_{-1,\g})$, $\g > -1$, are eigenfunctions of the differential operator $\CD_{\g}$,
defined in \eqref{eq:diff-op}. For the Sobolev orthogonal polynomials, we need $\g$ to be negative integers.
We first consider the action of the differential operator when $\g$ is a real number.

\subsection{Eigenfunctions of $\CD_{\g}$ for $\g \in \RR$}
 Recall that the operator $\CD_{\g}$ is given by
 $$
   \CD_{\g} =  t(1-t)\frac{\d^2}{\d t^2} + \big( d-1 - (d+\g)t \big) \frac{\d}{\d t}+ t^{-1} \Delta_0^{(\xi)}.
 $$
We start with a lemma for the action of this operator.

\begin{lem}
Let $p(t)$ be a polynomial in the $t$ variable and $Y(\xi)$ be a function defined on the unit sphere $\sph$.
For $k \geq 1$,
\begin{align}\label{eq:D_gamma}
\CD_{\gamma} \left [(1-t)^k\,p(t) Y(\xi) \right] \, & = (1-t)^k \CD_{\gamma+2k}\left [p(t) Y (\xi)\right] \\
& - k\,(1-t)^{k-1} [d-1-(d+\g+k-1)t] p(t)Y(\xi). \notag
\end{align}
\end{lem}

\begin{proof}
For $k=1$, a quick computation from the definition of $\CD_{\g}$ shows
\begin{align*}
\CD_{\g}[(1-t)  p(t) Y(\xi)]
 =\, & \big((1-t) \left[t(1-t)p''(t) +  [d-1-(d+\gamma+2)t]p'(t)\right] \\
    &  - [d-1-(d+\gamma)t]p(t)\big) Y(\xi)
     + t^{-1}\,(1-t) p(t) \, \Delta_0^{(\xi)}Y(\xi) \\
= \, & (1-t)\CD_{\gamma+2}[p(t)Y(\xi)] - [d-1-(d+\gamma)t]p(t)Y(\xi),
\end{align*}
which verifies \eqref{eq:D_gamma} for $k=1$. Assume that \eqref{eq:D_gamma} holds for $k\geq 1$.
Then
\begin{align*}
& \CD_{\gamma}[(1-t)^{k+1} p(t) Y(\xi)] = \CD_{\g} [(1-t)^k (1-t) p(t) Y(\xi)] \\
& \quad = (1-t)^k \CD_{\g+2k}[(1-t)p(t) Y(\xi)] \\
&\qquad - k\,(1-t)^{k-1} [d-1-(d+\gamma+k-1)t][(1-t)p(t) Y(\xi)]\\
& \quad =  (1-t)^k \left\{(1-t)\CD_{\g+2k+2}[p(t) Y(\xi)] - [d-1-(d+\gamma+2k)t]p(t) Y(\xi)\right\}\\
&\qquad - k\,(1-t)^{k} [d-1-(d+\gamma+k-1)t]p(t) Y(\xi)\\
& \quad = (1-t)^{k+1} \CD_{\g+2k+2}[p(t) Y(\xi)] - (k+1)(1-t)^k [d-1-(d+\gamma+k)t]p(t) Y(\xi),
\end{align*}
so that \eqref{eq:D_gamma} holds for $k+1$ and the proof is completed by induction.
\end{proof}

\begin{thm}\label{thm:Z}
For $0\le  j \le n$, let $Z_{j,n}(x,t) = p(t) Y(x)$ with $p$ being a polynomial of degree $j$ in one variable and
$Y$ be a solid harmonic polynomial in $\CH_{n-j}^{d,0}$. For $\gamma \in \RR$, the only polynomial $p$ for
which $Z_{j,n}$ satisfies
\begin{equation}\label{eq:Z}
     \CD_{\g} Z (x,t) = \lambda_n^{(\gamma)} Z(x,t),  \qquad \l_n^{(\g)} =  -n(n+\g+d-1),
\end{equation}
is the Jacobi polynomial $p(t) = c_j\,P_j^{(2n-2j+d-2, \gamma)}(1-2t)$, where $c_j$ is an appropriate constant,
if $2n+\gamma+d-r-1\neq 0$ for $0\leq r\leq j$.
\end{thm}

\begin{proof}
Without losing generality, we assume $p$ is a monic polynomial,
$$
p(t) = \sum_{i=0}^j \,(1-t)^i\,a_{j,i}, \quad a_{j,j} = 1.
$$
Our objective is to determine the coefficients $a_{j,i}$ such that $Z_j$ satisfies \eqref{eq:Z}. First, we observe
that \eqref{eq:Z} holds if $p(t) =1$. In this case $Y \in \CH_n^{d,0}$ and $Y(x) = t^n Y(\xi)$. Hence, taking
the derivative over $t$ and using the relation \eqref{eq:LB-operator}, a quick computation shows that
\begin{align} \label{eq:Delta0Y}
\CD_{\g} Y(x)= &\, \CD_{\g}[t^n Y(\xi)] \\
 = &\, n (n+d-2)Y(\xi) - m (m+\g+d-1) t^m Y(\xi) + t^{m-1}  \Delta_0 Y(\xi) \notag \\
 =  &\, -m(m+d+\g-1) t^m Y(\xi) = -m(m+d+\g-1) Y(x). \notag
\end{align}
For $j \ge 0$, we use the identity \eqref{eq:D_gamma} to obtain
\begin{align*}
& \CD_{\g} Z_j(x,t)  =  \sum_{i=0}^j  a_{j,i} \CD_{\g}\left[(1-t)^i t^{n-j} Y(\xi)\right] \\
& = \sum_{i=0}^j  a_{j,i} \left[(1-t)^i \CD_{\g+ 2i}[Y(x)] - i(1-t)^{i-1}[d-1-(d+\gamma+i-1)t]  Y(x)\right]
\end{align*}
and then apply \eqref{eq:Delta0Y} for $Y\in \CH_{n-j}^{d,0}$ and simplify to obtain
\begin{align*}
\CD_{\g} Z_j(x,t)  =
\, & \sum_{i=0}^{j} \,(1-t)^i \,a_{j,i} [\lambda_{n-j}^{(\gamma+2i)} - i\,(d+\gamma+i-1)]\,Y(x)\\
&+\sum_{i=0}^{j-1} (i+1)(1-t)^{i}\,(\gamma+i+1) \,a_{j,i+1}\,Y(x)\\
= \,  & (1-t)^j \,\lambda_n^{(\gamma)}\,Y(x) +\sum_{i=0}^{j-1} \,(1-t)^i \\
& \times \{a_{j,i} \,[\lambda_{n-j}^{(\gamma+2i)} - i(d+\gamma+i-1)] + a_{j,i+1}(i+1)(\gamma+i+1)\}Y(x),
\end{align*}
since $a_{j,j} = 1$. Thus, in order to have $Z_{j,n} (x,t)$ satisfying \eqref{eq:Z}, we must have
$$
a_{j,i} \,[\lambda_{n-j}^{(\gamma+2i)} - i(d+\gamma+i-1)] + a_{j,i+1}(i+1)(\gamma+i+1) = a_{j,i} \lambda_n^{(\gamma)}
$$
for $0 \le i \le j-1$, which leads to the recurrence relation
$$
a_{j,i}  = - \frac{(i+1)(\gamma+i+1)}{(j-i)(2n+\gamma+d-1-j+i)}a_{j,i+1} , \quad i=0, 1, \ldots, j-1.
$$
Solving this recurrence relation and using $a_{j,j}  = 1$, we obtain
\begin{equation}\label{def:a}
a_{j,i} = (-1)^{j-i}\frac{(i+1)_{j-i}(\gamma+i+1)_{j-i}}{(j-i)!(2n+\gamma+d- 1-j+i)_{j-i}}, \quad 0\leq i \leq j.
\end{equation}
Using the identity $(a+i)_{j-i} = (a)_j/(a)_i$ and $(j-i)! = (-1)^i j! / (-j)_i$, we can rewrite it as
$$
 a_{j,i} = (-1)^j c_j \frac{(-j)_i (2n+\g+d-j-1)_i}{i! (\g+1)_i}, \quad c_j= \frac{(\g+1)_j}{(2n+\g+d-j-1)_j},
$$
which is well-defined if $(2n+\g+d-j-1)_j \ne 0$. Consequently, the polynomial $p$ must be given by
\begin{align*}
 p(x) \,& = (-1)^j c_j
    \phantom{i}_2F_1\left ( \begin{matrix} -j, 2n+\g+d-j-1 \\ \g+1 \end{matrix}; 1-t \right) \\
& = (-1)^jc_j P_j^{(\g, 2n-2j+d-2)}(2t-1) =
      c_j P_j^{(2n-2j+d-2, \g)}(1-2t),
\end{align*}
which completes the proof.
\end{proof}

For $\g > -1$, the theorem recovers Theorem \ref{thm:Jacobi-DE-V0} established in \cite{X20}. We discuss the case
when $\g  = -s$ and $s \in \NN$ in the next subsection.

\subsection{Eigenfunctions for $\CD_{-s}$}
As a corollary of Theorem \ref{thm:Z}, we obtain the following result for $\CD_{-s}$.

\begin{prop}
Let $s \in \NN$ and $n \in \NN_0$. Then the polynomials $\sS_{m,\ell}^{n, (-1,-s)}$ are eigenfunctions
of $\CD_{-s}$ if and only if $m \le n-s$ and $m =n$.
\end{prop}

\begin{proof}
By \eqref{eq:Jn(a,-s)=Pn}, the polynomials $\sS_{m,-\ell}^{n,(-1,-s)}$ can be written as
$$
\sS_{m,\ell}^{n,(-1,-s)}(x,y) = c_{n-m} P_{n-m}^{(2m+d-2,-s)}(1-2t) Y_\ell^m(x), \quad n-m \ge s,
$$
so that Theorem \ref{thm:Z} applies for $n - m \ge s$, thus $\sS_{m,\ell}^{n, (-1,-s)}$ satisfies
\eqref{eq:Z}. If $n = m$, then $\sS_{n, \ell}^{n, (-1,-s)}(x,t) = Y_\ell^n(x)$. As a result,
\eqref{eq:Z} holds with $p(z) =1$. For $1 \le n-m \le s-1$, however, the polynomial
$J_{n-m}^{(2m+d-2,-s)}(t) = (1-t)^{n-m}/(n-m)!$ is not the Jacobi polynomial $P_{n-m}^{(2m+d-2,-s)}(t)$,
hence $\sS_{n, \ell}^{n, (-1,-s)}$ is not a solution of \eqref{eq:Z}.
\end{proof}

In the case $s =1$, the set $\{m: n -m \ge s\} \cup \{n\} = \{m: 1 \le m \le n\}$. Hence, we obtain the
following corollary.

\begin{thm}
For $\g = -1$, all elements in $\CV_n^d(\sw_{-1,-1})$ are eigenfunctions of $\CD_{-1}$; that is,
$$
   \CD_{-1} Z = - n  (n+ d-2) Z, \qquad \forall Z \in \CV_n^d(\sw_{-1,-1}).
$$
Moreover, the space $\CV_n^d(\sw_{-1,-1})$ satisfies a decomposition
$$
\CV_n^d(\sw_{-1,-1}) =  \CH_n^{d,0} \cup (1-t) \CV_{n-1}^d(\sw_{-1,1}), \qquad n =0,1,\ldots.
$$
\end{thm}

\begin{proof}
We only need to prove the decomposition. If $m \le n-1$, it follows from from \eqref{eq:Jn(a,-s)=Pn} that
$\sS_{m,\ell}^{n,(\b,-1)}  = c (1-t) \sS_{m,\ell}^{n-1,(\b,1)} \in (1-t) \CV_{n-1}^d(\sw_{-1,1})$, whereas if $m = n$,
then $\sS_{n,\ell}^{n,(\b,-s)} = Y_\ell^n \in \CH_n^{d,0}$.
\end{proof}

The theorem shows, in particular, that $\CD_{\g}$ has the complete eigenspaces if $\g \ge -1$.
For $s = 2,3,\ldots$, however, not every element of the space $\CV_n^d(\sw_{-1,-s})$ is an eigenfunction of
$\CD_{ -s}$ by Theorem \ref{thm:Z}. We can, however, identify the eigenspaces of the operator as follows.

\begin{thm}
For $s = 2,3,\ldots$, define
\begin{align*}
\CU_n^d(\sw_{-1,-s}):= \CH_{n}^{d,0}  \bigcup_{j=1}^{s-1}P_j^{(2n-2j+d-2,-s)}(1-2t)\CH_{n-j}^{d,0}
     \cup (1-t)^s \CV_{n-s}^d (\sw_{-1,s}).
\end{align*}
Then $\CU_n^d(\sw_{-1,-s})$ is the eigenspace of $\CD_{-s}$; more precisely,
$\CD_{-s} Z=\lambda_n^{(-s)}Z$ for all $Z \in \CU_n^d(\sw_{-1,-s})$. Moreover, the space satisfies
$$
  \dim \CU_n^d(\sw_{-1,-s}) = \binom{n+d-1}{d}
$$
if $-2n-d+j+s+2$ is not a positive integer between 1 and $j$ for $1 \le j \le s-1$ and $n \ge j$.
\end{thm}

\begin{proof}
As in the case of $s = 1$, using \eqref{eq:Jn(a,-s)=Pn}, the statement on the eigenspace follows readily
from Theorem \ref{thm:Z}. As it is shown for the dimension of $\CV_n^d(\sw_{\b,-s})$, the space
$\CU_n^d(\sw_{-1,-s})$ has the full dimension $\binom{n+d-1}{d}$ if the Jacobi polynomials
$P_j^{(2n-2j+d-2,-s)}$ does not have the degree reduction, which holds if $2n-j+d-2-s +k =0$ for a
certain integer $k$, $1 \le k \le j$, by \cite[p. 64]{Sz}.
\end{proof}

The theorem shows that the operator $\CD_{-s}$ has a complete basis of polynomials as eigenfunctions
only if the restriction on $-2n-d+j+s+2$ as stated holds.

\begin{exam}
If $s = 2$, then the space $\CU_n^d(\sw_{-1,-2})$ is given by
$$
  \CU_n^d(\sw_{-1,-2}) = \CH_n^d \cup  \left(1+ (2n+d-4)(1-t) \right) \CH_{n-1}^d
   \cup (1-t)^s \CV_{n-2}^d (\sw_{-1,2})
$$
and it has the full dimension if $-2n-d+5 \ne 1$ or $2n-d+4 \ne 0$ for $n \ge 1$. Thus, $\CD_{-2}$ has a complete
basis of polynomials as eigenfunctions if $d$ is odd. In the case that $d$ is even, $ \CU_n^d(\sw_{-1,-2})$ is
not well-defined for $n = d/2-2$.
\end{exam}

We note that the space $\CU_n^d(\sw_{-1,-s})$ does not coincide with $\CV_n^d(\sw_{-1,-s})$ for $s =2,3,\ldots$.
In the case of $s = -2$ and $d$ is odd, one may ask if there is a Sobolev inner product for which the polynomials
in the space $\CU_n^d(\sw_{-1,-s})$ are orthogonal.

\bigskip

The authors have no relevant financial or non-financial interests to disclose.

\end{document}